\documentclass[11pt,a4paper]{amsart}
\usepackage{graphicx,multirow,array,amsmath,amssymb,color,enumitem,subfig}

\newtheorem{theorem}{Theorem}
\newtheorem{proposition}[theorem]{Proposition}
\newtheorem{lemma}[theorem]{Lemma}

\newcommand{\adj}{\!\thicksim\!}
\newcommand{\nadj}{\!\nsim\!}

\begin{document}
	
\title{Four triangle-free intrinsically knotted graphs with 22 edges}

\author[H. Kim]{Hyoungjun Kim}
\address{Department of Mathematics Education, Kyungpook National University, Daegu 41566, Korea}
\email{kimhjun@knu.ac.kr}

\author[T.W. Mattman]{Thomas W.~Mattman}
\address{Department of Mathematics and Statistics, California State University, Chico, Chico CA 95929-0525, USA}
\email{TMattman@CSUChico.edu}

\author[S. Oh]{Seungsang Oh}
\address{Department of Mathematics, Korea University, Seoul 02841, Korea}
\email{seungsang@korea.ac.kr}

\dedicatory{Dedicated to the memory of Professor Hwa Jeong Lee.}
	

\begin{abstract}
An intrinsically knotted graph is one for which every spatial embedding contains a nontrivially knotted cycle.
Classifying such graphs is a central problem in spatial graph theory.
It is known that every intrinsically knotted graph has at least 21 edges, and the case of 21 edges has been completely resolved.
For 22 edges, however, the classification remains incomplete.
In particular, exactly eight triangle-free examples with a vertex of degree at least 5 are known, leaving only the case in which all vertices have degree 3 or 4.
In this paper, we introduce a method for detecting intrinsic knottedness based on induced subgraphs obtained by deleting pairs of vertices.
Using this method, we classify all triangle-free intrinsically knotted graphs with 22 edges having eight vertices of degree~4 and four of degree~3.
We prove that there are exactly four: Cousins 43, 105, and 109 in the $E_9\!+\!e$ family and the graph $H_{12}\! +\! e$ in the $H_9\!+\!e$ family.
\end{abstract}

\maketitle

\section{Introduction} \label{sec:intro}

A graph $G$ is {\em intrinsically knotted\/} (IK) if every embedding of $G$ in $\mathbb{R}^3$ contains a nontrivially knotted cycle.
A graph $H$ is a {\em minor\/} of a graph $G$ if $H$ can be 
obtained from $G$ by contracting or deleting
zero or more
edges.
If $G$ is IK and has no proper IK minor, then $G$ is said to be {\em minor minimal intrinsically knotted\/} (MMIK).
By Robertson and Seymour's Graph Minor Theorem~\cite{RS}, there are only finitely many MMIK graphs, but determining the complete set remains an open problem.

A $\nabla Y$ {\em move\/} is an operation on a graph that removes all edges of a 3-cycle $abc$, introduces a new vertex $v$, and adds three edges $va$, $vb$ and $vc$.
The reverse operation is called a $Y \nabla$ {\em move\/}; see Figure~\ref{fig:Y}.
The $\nabla Y$ move preserves intrinsic knottedness.
Sachs~\cite{S} observed this for intrinsic linking, and the proof for intrinsic knotting is analogous.
Thus, in the study of IK graphs, it is natural to focus on triangle-free graphs.

\begin{figure}[ht]
\includegraphics{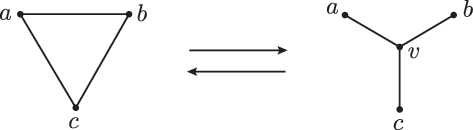}
\caption{$\nabla Y$ and $Y \nabla$ moves}
\label{fig:Y}
\end{figure}

Johnson, Kidwell, and Michael~\cite{JKM} and the second author~\cite{M} showed that every IK graph has at least 21 edges.
Working independently, two groups \cite{BM,LKLO} proved that $K_7$ and the 13 graphs obtained from $K_7$ by $\nabla Y$ moves are the only IK graphs with 21 edges.
Among these, $H_{12}$ and $C_{14}$ are triangle-free.

More generally, all known MMIK graphs and triangle-free IK graphs
with at most 22 edges 
belong to one of four families, namely the $K_{3,3,1,1}$, $E_9\!+\!e$, $H_8\!+\!e$, and $H_9\!+\!e$ families, 
see
Figure~\ref{fig:4family}.
Goldberg et al.~\cite{GMN} studied the $K_{3,3,1,1}$ family, which contains 56 graphs, and the $E_9\!+\!e$ family, which contains 110 graphs.
In~\cite{KM}, we introduced the $H_8\!+\!e$ family, with 125 graphs, and the $H_9\!+\!e$ family, with 5 graphs.
Flapan et al.~\cite{FMMNN} described a computer search showing that there are 92 MMIK graphs with 22 edges, all of which belong to one of these four families.
More recently, we used computational methods to show that there are exactly three MMIK graphs with 23 edges~\cite{KM}.

\begin{figure}[ht] 
\includegraphics{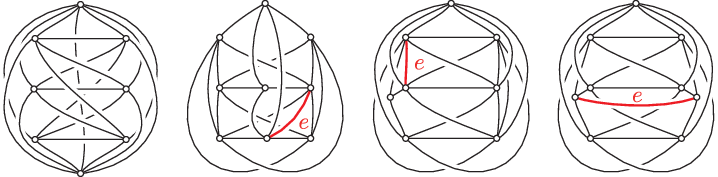}
\caption{$K_{3,3,1,1}$, $E_9\!+\!e$, $H_8\!+\!e$, and $H_9\!+\!e$}
\label{fig:4family}
\end{figure}

In earlier work \cite{KLLMO,KMO2}, we completed the classification of triangle-free IK graphs with 22 edges that have at least one vertex of degree at least~5.
More precisely, \cite{KLLMO} showed that there are exactly three such graphs with at least two degree-5 vertices (Cousins 94 and 110 in the $E_9\!+\!e$ family and $M_{11}$ in the $H_8\!+\!e$ family) and none with maximum degree greater than~5.
In~\cite{KMO2}, we identified exactly five such graphs with a unique degree-5 vertex: Cousin 29 in the $K_{3,3,1,1}$ family, Cousins 97 and 99 in the $E_9\!+\!e$ family, and $U_{12}$ and $U'_{12}$ in the $H_8\!+\!e$ family.

Since a MMIK graph has minimum degree at least~3, the remaining triangle-free IK graphs with 22 edges must have all vertices 
in $V_3$ or $V_4$, 
where $V_d$ denotes the vertices of degree $d$.
Such graphs fall into four types according to $(|V_4|,|V_3|)$: $(11,0)$, $(8,4)$, $(5,8)$, and $(2,12)$.
In this paper, we study the second type, namely $(8,4)$.

\begin{theorem} \label{thm:main}
There are exactly four triangle-free intrinsically knotted graphs 
with 22 edges having eight vertices of degree~4 and four vertices of degree~3: Cousins 43, 105, and 109 in the $E_9\!+\!e$ family and 
the graph $H_{12}\! +\! e$ in the $H_9\!+\!e$ family (see Figure~\ref{fig:main}).
\end{theorem}

\begin{figure}[ht] 
\includegraphics[scale=1]{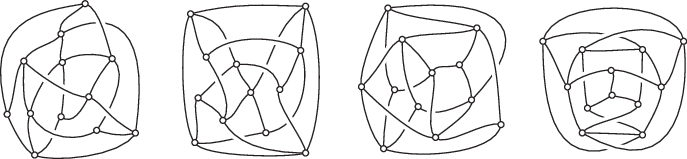}
\caption{Four triangle-free IK graphs of type $(8,4)$: Cousins 43, 105, and 109 in the $E_9\!+\!e$ family and 
the graph $H_{12}\!+\! e$ in the $H_9\!+\!e$ family}
\label{fig:main}
\end{figure}

As shown in \cite{GMN} and \cite{KM}, all graphs in the 
$E_9\!+\!e$ and $H_9\!+\!e$ families are IK, including
the four graphs in our theorem. Our proof amounts to verifying that, 
among the 728 simple triangle-free graphs of type $(8,4)$, 
those four are the only ones that are not $2$-apex.

\section{Preliminaries and degree sequence reduction} \label{sec:main}

Let $G=(V,E)$ be a graph with vertex set $V$ and edge set $E$.
We say that a vertex $a$ is adjacent to another vertex $b$, denoted by $a \adj b$, if there is an edge joining them; otherwise, we write $a \nadj b$.
Let $V(a)=\{b\! \in\! V\, |\, b \adj a \}$ denote the neighborhood of $a$.
We write $\Delta(G)$ and $\delta(G)$ for the maximum and minimum degrees of $G$, respectively.

Throughout the paper, let $G$ be a connected triangle-free IK graph with 22 edges, eight vertices of degree~4, and four vertices of degree~3.
For distinct vertices $a$ and $b$, let $G_{a,b}$ denote the graph obtained from $G$ by deleting the vertices $a$ and $b$, together with the interiors of all edges incident to them, and then deleting any resulting isolated vertices.
Note that $G_{a,b}$ may still contain vertices of degree~1.

\begin{lemma} \label{lem:choice_ab}
Let $G$ be a graph with eight degree-4 vertices and four degree-3 vertices.
Then there exists a pair $\{a,b\}$ of non-adjacent degree-4 vertices such that $V(a)\cap V(b)$ contains no degree-3 vertex.
\end{lemma}

\begin{proof}
There must exist a degree-4 vertex $a$ having at most one neighbor of degree~3.
Indeed, otherwise every vertex in $V_4$ would have at least two neighbors in $V_3$, so there are at least $16$ edges joining $V_4$ to $V_3$.
But the four vertices in $V_3$ can be incident to at most $12$ such edges, a contradiction.

If $a$ has no neighbor of degree~3, then all four neighbors of $a$ lie in $V_4$, and so we may choose $b \in V_4 \setminus (V(a)\cup\{a\})$.
If $a$ has exactly one degree-3 neighbor $c$, then 
$|(V(a)\cup V(c)) \cap V_4| \le 6$, 
and hence we may choose
$b \in V_4 \setminus (V(a)\cup V(c))$.
In either case, $a$ and $b$ are non-adjacent degree-4 vertices, and $V(a)\cap V(b)$ contains no degree-$3$ vertex.
\end{proof}

For the remainder of the paper, fix a pair $\{a,b\}$ of vertices as in Lemma~\ref{lem:choice_ab}.
Then $G_{a,b}$ is a triangle-free graph with 10 vertices and 14 edges.
Moreover, by the choice of $a$ and $b$, it has no vertex of degree~1.
When $G_{a,b}$ has degree-2 vertices, we often remove each such vertex $x$ by {\em smoothing\/}, that is, by contracting an edge incident to $x$ so that the two neighbors of $x$ become adjacent.
The resulting graph, denoted by $\widehat{G}_{a,b}$, 
may contain 3-cycles, double edges, or loops.

A graph is called {\em 2-apex\/} if it can be made planar by deleting two vertices.
It is known that no $2$-apex graph is IK~\cite{BBFFHL,OT}.
Since $G$ is IK, neither $G_{a,b}$ nor $\widehat{G}_{a,b}$ can be planar.
Due to the Euler characteristic, each has a $K_{3,3}$ minor but no $K_5$ minor.
We classify the possible graphs $G_{a,b}$ according to their degree sequence $[\,|V_4|,|V_3|,|V_2|\,]$.

\begin{proposition}  \label{prop:deg_seq}
The possible degree sequences of $G_{a,b}$ are $[0,8,2]$, $[1,6,3]$, and $[2,4,4]$. 
\end{proposition}

\begin{proof}
Since $G_{a,b}$ has 10 vertices and 14 edges, we have $|V_4|+|V_3|+|V_2|=10$ and $4|V_4|+3|V_3|+2|V_2|=28$.
Subtracting twice the first equation from the second gives $2|V_4|+1|V_3|=8$.
If $|V_4|=0$, 1, or 2, we obtain the three degree sequences listed above.

If instead $|V_4|=3$ or $4$, then the degree sequences are $[3,2,5]$ or $[4,0,6]$, respectively.
After smoothing all degree-2 vertices, we obtain either a multigraph on five vertices with nine edges or a multigraph on four vertices.
Both are planar, contradicting the nonplanarity of $G_{a,b}$.
Hence these cases cannot occur.
\end{proof}

For the four graphs in Theorem~\ref{thm:main}, the associated graphs $G_{a,b}$ arise as follows.
Cousin $43$ yields five distinct graphs $G_{a,b}$: four of type $[2,4,4]$ and one of type $[1,6,3]$.
Each of Cousins $105$ and $109$, as well as $H_{12}\!+\!e$, yields two graphs $G_{a,b}$, both of type $[2,4,4]$.

\section{Degree sequence $[0,8,2]$} \label{sec:082}

Recall that $G$ is a connected triangle-free IK graph with eight degree-$4$ vertices and four degree-$3$ vertices, and that $\{a,b\}$ is a pair of non-adjacent degree-$4$ vertices such that $V(a)\cap V(b)$ contains no degree-$3$ vertex.
In this section, we consider the case in which $G_{a,b}$ has degree sequence $[0,8,2]$.
Let $d_1$ and $d_2$ denote the two degree-$2$ vertices of $G_{a,b}$.
As noted at the end of Section~\ref{sec:main}, we will show that there is no triangle-free IK graph $G$ that produces such a graph $G_{a,b}$.

After smoothing $d_1$ and $d_2$, we obtain a nonplanar graph $\widehat{G}_{a,b}$ with eight degree-$3$ vertices and $12$ edges.
Since $\widehat{G}_{a,b}$ is nonplanar, it has a $K_{3,3}$ minor.
Hence we may assume that there are three edges $e_1$, $e_2$, and $e_3$ such that $K_{3,3}$ is obtained from $\widehat{G}_{a,b}$ by deleting $e_1$ and then contracting $e_2$ and $e_3$, as follows:
$$\begin{array}{ccccc}
\widehat{G}_{a,b} & \rightarrow & \! \widehat{G}_{a,b} \setminus e_1 \! & \! \rightarrow \! & \! (\widehat{G}_{a,b} \setminus e_1) / \{e_2,e_3\}\, \simeq\, K_{3,3} \\[3pt]
\! {[}0,8,0{]} \! & & {[}0,6,2{]} & & {[}0,6,0{]} \quad \quad \quad \\
\end{array}$$

Once the edge $e_1$ to be deleted from $\widehat{G}_{a,b}$ is fixed, one contraction is performed at each of its two endpoints in $\widehat{G}_{a,b}\setminus e_1$. 
According to the positions of these two endpoints, three cases arise: they may lie on a single edge of $K_{3,3}$, on two adjacent edges, or on two non-adjacent edges. 
By the symmetry of $K_{3,3}$, there are three possible forms of $\widehat{G}_{a,b}$, shown in Figure~\ref{fig:082}(a), (b), and (c), respectively. 
We treat these three cases in the following subsections.

\begin{figure}[ht] 
\includegraphics[scale=1]{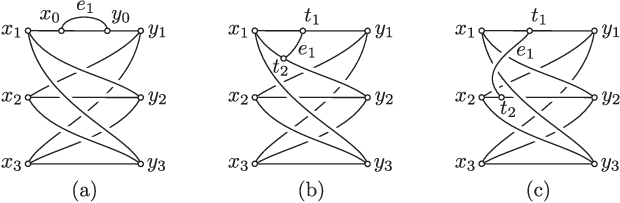}
\caption{Three possible graphs $\widehat{G}_{a,b}$ when $G_{a,b}$ has degree sequence [0,8,2].}
\label{fig:082}
\end{figure}

\subsection{$\widehat{G}_{a,b}$ has double edges} \

We now consider the case in which $\widehat{G}_{a,b}$ is the multigraph shown in Figure~\ref{fig:082}(a), using the labeling indicated there.
Since $G_{a,b}$ is triangle-free, the degree-$2$ vertices $d_1$ and $d_2$ must lie on the two doubled edges.
There are two possibilities: either (1) $d_1$ and $d_2$ lie on different doubled edges, or (2) both lie on the same doubled edge.
In both cases, we assume that $d_1 \adj x_0$ and $a \adj d_1$ in the original graph $G$.
Clearly, $a \nadj x_0$.

Suppose that either $a \adj d_2$ in case (1) or $a \adj y_0$ in case (2).
Let $z_1,z_2$ denote the remaining two neighbors of $a$. 
Since $G$ is triangle-free, 
in case (2), $\{z_1,z_2\}$ is either $\{y_2,y_3\}$ or two from $\{x_1,x_2,x_3\}$. In case (1), $\{z_1, z_2\}$
is $\{x_1,y_1\}$ or else two from $\{x_1,x_2,x_3\}$
or $\{y_1,y_2,y_3\}$.
In each case, a direct inspection of the labeled graph shows that $G_{z_1,z_2}$ is planar, contradicting the fact that $G$ is not $2$-apex.
Thus $a \nadj d_2,y_0$.

Therefore, we may assume that $a \adj x_1,x_2,x_3$ (or symmetrically, $a \adj y_1,y_2,y_3$).
Again by direct inspection, $G_{b,x_2}$ is planar.
Therefore the double-edge case cannot occur.

\subsection{$\widehat{G}_{a,b}$ has a 3-cycle} \

We consider the case in which $\widehat{G}_{a,b}$ contains the $3$-cycle $C=(x_1t_1t_2)$ as shown in Figure~\ref{fig:082}(b).
Since $G_{a,b}$ is triangle-free, we may assume that $d_1$ lies on $C$.
We further assume that $a \adj d_1$ in $G$.

We first observe that $d_2 \nadj x_2,x_3$.
Indeed, suppose for example that $d_2 \adj x_2$.
Then, unless $a \adj x_3$, 
one of $G_{b,y_1}$ and $G_{b,y_2}$ is planar,
since the four edges incident to $a$ can be embedded in the plane without crossing.
If instead $a \adj x_3$, then $b \nadj x_3$, since $x_3$ has degree at most $4$ in $G$.
But then 
$G_{a,y_1}$ or $G_{a,y_2}$
is planar, a contradiction.
Hence $d_2 \nadj x_2,x_3$.

Now consider $G_{b,y_1}$.
For this graph to be nonplanar, we must have $a \adj x_2,x_3$.
In this case, for $G_{x_2,x_3}$ to be nonplanar, 
the vertex $d_2$ must lie either on the edge $(x_1y_3)$ or on the edge $(t_iy_i)$ for some $i=1,2$, and moreover $a \adj d_2$. 
By symmetry, it suffices to consider only the cases where $d_2$ lies on $(x_1y_3)$ or on $(t_2y_2)$. 
Finally, consider $G_{a,y_1}$. 
Since $b \nadj x_2,x_3$, in either case, the graph $G_{a,y_1}$ is 
planar, which is a contradiction, unless 
$d_1$ subdivides $(x_1t_2)$ and 
$b \adj d_1, t_1$.
On the other hand, when $d_1$ lies on $(x_1t_2)$ and both $d_1$ and $t_1$ are in $N(b)$, then $G_{b,x_3}$ is planar
and we again have a contradiction.

The 
contradiction
shows that $\widehat{G}_{a,b}$ cannot contain a $3$-cycle.


\subsection{$\widehat{G}_{a,b}$ is the M\"obius ladder $M_4$} \

The graph in Figure~\ref{fig:082}(c) is isomorphic to the M\"obius ladder $M_4$.
So we label the vertices as in Figure~\ref{fig:M4}.
Recall that $M_4$ consists of the $8$-cycle $(v_1v_2v_3v_4v_5v_6v_7v_8)$ together with the four edges, called rungs, $(v_1v_5)$, $(v_2v_6)$, $(v_3v_7)$ and $(v_4v_8)$. In the labeling of Figure~\ref{fig:082}(c), the $8$-cycle $(v_1v_2v_3v_4v_5v_6v_7v_8)$ corresponds to $( x_1t_1y_1 x_3 y_2 t_2 x_2 y_3 )$.

\begin{figure}[ht] 
\includegraphics[scale=1]{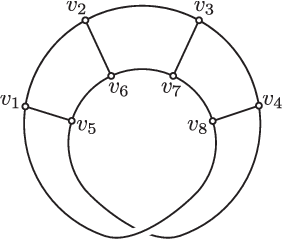}
\caption{The M\"obius ladder $M_4$.}
\label{fig:M4}
\end{figure}

We first claim that for any three consecutive vertices $v_i$, $v_{i+1}$, and $v_{i+2}$ on the $8$-cycle, at least one of them is adjacent to a degree-$2$ vertex of $G_{a,b}$.
Suppose otherwise that $v_1$, $v_2$, and $v_3$ are not adjacent to any degree-$2$ vertex.
Then consider $G_{v_1,v_3}$.
Since $v_2$ has degree $3$ in $G_{a,b}$, it is adjacent to at most one of $a$ and $b$ in $G$.
As $v_2$ is not adjacent to any degree-$2$ vertex, it follows that $G_{v_1,v_3}$ is planar, a contradiction.

Now suppose that a degree-$2$ vertex lies on the $8$-cycle, say $d_1 \adj v_1,v_2$.
By the claim, the other degree-$2$ vertex $d_2$ must be adjacent to at least one vertex of each of the sets $\{v_3,v_4,v_5\}$, $\{v_4,v_5,v_6\}$, $\{v_5,v_6,v_7\}$, and $\{v_6,v_7,v_8\}$.
Hence $d_2 \adj v_5,v_6$.
Therefore the four vertices $v_3$, $v_4$, $v_7$, and $v_8$ are not adjacent to any degree-$2$ vertex.
Now consider $G_{v_3,v_8}$.
Both $v_4$ and $v_7$ have degree $3$ in $G_{a,b}$, so each is adjacent to at most one of $a$ and $b$.
Since neither is adjacent to a degree-$2$ vertex, $G_{v_3,v_8}$ is planar, again a contradiction.
Thus both degree-$2$ vertices must lie on rungs.

If the two degree-$2$ vertices lie on the same rung or on two consecutive rungs, say $(v_1v_5)$ and $(v_2v_6)$, then the same argument as above shows that $G_{v_3,v_8}$ is planar.
Therefore we may assume that they lie on two non-consecutive rungs, namely $d_1 \adj v_1,v_5$ and $d_2 \adj v_3,v_7$.

First suppose that $a \adj d_1,d_2$.
Then, without loss of generality, we may assume that $a \adj v_2,v_4$.
If $b \adj d_1,d_2$, then $b \adj v_6,v_8$, and in this case $G_{v_1,v_4}$ is planar.
On the other hand, if $b \nadj d_1,d_2$, then $G_{v_1,v_3}$ is planar, since $b \nadj v_2$.
Thus it is not true that
$a \adj d_1,d_2$.
Similarly, we can exclude the case where $a \nadj d_1,d_2$ forces $b \adj d_1,d_2$.

Therefore, without loss of generality, we may assume that $a \adj d_1$ and $a \nadj d_2$, so that $b \adj d_2$.
Since $G$ is triangle-free, we may assume that
$a \adj v_2,v_4,v_7$. 
Then $b \adj v_6,v_8,d_1$.
But now $G_{v_2,v_4}$ is planar.
Hence $G$ is $2$-apex, a contradiction.
Therefore this case cannot occur.

\section{Degree sequence $[1,6,3]$} \label{sec:163}

In this section, we consider the second case, in which $G_{a,b}$ has degree sequence $[1,6,3]$.
Let $d_1$, $d_2$, and $d_3$ denote the degree-$2$ vertices of $G_{a,b}$.
As noted at the end of Section~\ref{sec:main}, we will show that the only graph $G$ yielding such a graph $G_{a,b}$ is Cousin $43$ in the $E_9\!+\!e$ family.

After smoothing $d_1$, $d_2$, and $d_3$, we obtain a nonplanar graph $\widehat{G}_{a,b}$ with one degree-$4$ vertex and six degree-$3$ vertices, 
having $11$ edges.
Since three edge contractions are used, $\widehat{G}_{a,b}$ may contain a loop or double edges, even though $G_{a,b}$ is a simple triangle-free graph.
Because $\widehat{G}_{a,b}$ is nonplanar, it has a $K_{3,3}$ minor.
Hence we may assume that there are two edges $e_1$ and $e_2$ such that $K_{3,3}$ is obtained from $\widehat{G}_{a,b}$ by deleting $e_1$ and then contracting $e_2$, as follows:
$$\begin{array}{ccccc}
\widehat{G}_{a,b} & \rightarrow & \! \widehat{G}_{a,b} \setminus e_1 \! & \! \rightarrow \! & \! (\widehat{G}_{a,b} \setminus e_1) / e_2\, \simeq\, K_{3,3} \\[3pt]
\! {[}1,6,0{]} \! & & & & {[}0,6,0{]} \quad \quad \quad \quad \\
\end{array}$$

Once the deleted edge $e_1$ is fixed, the contraction is performed at an endpoint $z$ of $e_1$ that has degree $1$ or $2$ in $\widehat{G}_{a,b}\setminus e_1$.
According to the type and position of $e_1$, four cases arise, as shown in Figure~\ref{fig:163}(a)--(d), respectively:
\begin{enumerate}
\item[(a)] $e_1$ is a loop incident to a degree-$3$ vertex;
\item[(b)] $e_1$ is a loop incident to a degree-$4$ vertex;
\item[(c)] $e_1$ is a double edge;
\item[(d)] $e_1$ is neither a loop nor a double edge.
\end{enumerate}
In the first three cases, $\widehat{G}_{a,b}$ is a multigraph, whereas in the last case it is a simple graph.

\begin{figure}[ht] 
\includegraphics[scale=1]{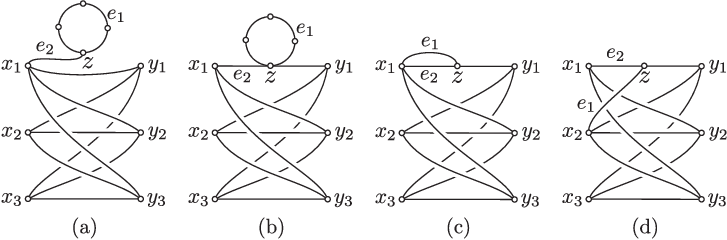}
\caption{Four possible graphs $\widehat{G}_{a,b}$ when $G_{a,b}$ has degree sequence $[1,6,3]$}
\label{fig:163}
\end{figure}

\subsection{$\widehat{G}_{a,b}$ is a multigraph}\label{subsec:multi}\ 

We first consider case (a) shown in Figure~\ref{fig:163}(a), using the labeling indicated there.
Since $G_{a,b}$ is a simple triangle-free graph, the three degree-$2$ vertices $d_1$, $d_2$, and $d_3$ lie on the loop $e_1$ in this order.
Again, since $G$ is triangle-free, we may assume that $a \adj d_1,d_3$ and $b \adj d_2$.
Then $a$ is either adjacent to both $x_2$ and $x_3$, or to two vertices among $y_1$, $y_2$, and $y_3$.
In either case, $G_{b,x_1}$ is planar.

Now consider case (b) shown in Figure~\ref{fig:163}(b).
As in the previous case, we may assume that $a \adj d_1,d_3$ and $b \adj d_2$.
Without loss of generality, we may assume that $b \adj y_1,y_2,y_3$, since $b$ has three remaining neighbors.
This implies that $a$ is adjacent to two vertices among $x_1$, $x_2$, and $x_3$, say $x_i$ and $x_j$.
Then $G_{x_i,x_j}$ is planar.

Finally, consider case (c) shown in Figure~\ref{fig:163}(c), where $\widehat{G}_{a,b}$ has double edges $e_1$ and $e_2$.
At least two of the degree-$2$ vertices lie on these double edges.
Hence at least one of $x_2$ and $x_3$ is not adjacent to any degree-$2$ vertex, say $x_2$.
Since $G$ has maximum degree $4$, one of $a$ and $b$ is not adjacent to $x_2$, 
assume
$a \nadj x_2$.
Then $G_{b,y_1}$ is planar.

Thus none of the multigraph cases can occur.

\subsection{$\widehat{G}_{a,b}$ is a simple graph}\label{subsec:simple}\ 

In this subsection, we consider case (d) shown in Figure~\ref{fig:163}(d), where $e_1$ joins a degree-$4$ vertex and a degree-$3$ vertex that are not adjacent in $K_{3,3}$.
For convenience, we redraw $\widehat{G}_{a,b}$ in the form shown in Figure~\ref{fig:sg711}, with the labeling indicated there.
Clearly, a degree-$2$ vertex, say $d_3$, lies on the $3$-cycle $(wx_1x_2)$.

\begin{figure}[ht] 
\includegraphics{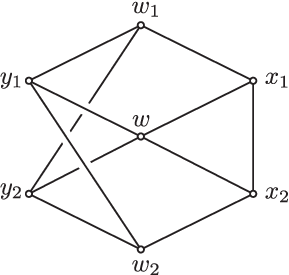}
\caption{A nonplanar simple graph.}
\label{fig:sg711}
\end{figure}

In this case, we determine $G$ directly through a sequence of restrictions.
Suppose first that $y_1$ is not adjacent to any degree-$2$ vertex.
Since $G$ has maximum degree $4$, one of $a$ and $b$ is not adjacent to $y_1$; say $a \nadj y_1$.
Then $G_{b,w}$ is planar.
Therefore, we may assume that $y_1 \adj d_1$, and similarly that $y_2 \adj d_2$.

We next claim that each of $a$ and $b$ is adjacent to one vertex from each of the pairs $\{y_1,d_1\}$ and $\{y_2,d_2\}$, but not to both $y_1$ and $y_2$.
Indeed, if $a \nadj d_1,y_1$, then again $G_{b,w}$ is planar.
Now suppose that $a \adj y_1,y_2$.
If $a \adj w_1$, then $w_1 \adj d_1,d_2$, and furthermore $b \adj d_1,d_2$ since $a \nadj d_1,d_2$.
Then $G_{y_1,y_2}$ is planar, since $a \nadj x_1,w,w_2$ and $b \nadj w$.
Hence $a \nadj w_1,w_2$.
It follows that $a$ is adjacent to exactly two vertices among $x_1$, $x_2$, and $d_3$, and then $G_{y_1,y_2}$ is planar.
This contradiction proves the claim.

We now show that neither $w_1$ nor $w_2$ is adjacent to both $d_1$ and $d_2$.
Suppose, for a contradiction, that $w_1 \adj d_1,d_2$.
If $a \adj y_1$, then $a \adj d_2$ and $b \adj d_1$.
Thus $a$ is adjacent to exactly two vertices among $x_1$, $x_2$, and $d_3$.
Then $G_{y_1,d_2}$ is planar, except possibly when $d_3 \adj b,w,x_2$.
If $d_3 \adj b,w,x_2$, then $a \adj d_3,x_1$, and hence $G_{x_1,y_1}$ is planar.
Thus both $a$ and $b$ must be adjacent to both $d_1$ and $d_2$, and so each is adjacent to exactly two vertices among $d_3$, $w_2$, $x_1$, and $x_2$.
Since $G$ has maximum degree $4$, we may assume that $a \nadj w_2$ and $b \adj w_2$.
Then $G_{d_1,d_2}$ is planar regardless of the position of $d_3$, a contradiction.

We claim next that $w$ is adjacent to neither $d_1$ nor $d_2$.
First suppose, for a contradiction, that $w \adj d_1,d_2$.
If $a \adj y_1$, then $a \adj d_2$ and $b \adj d_1$.
Thus $a$ is adjacent to exactly two vertices among $x_1$, $x_2$, and $d_3$.
If also $b \adj y_2$, then similarly $b$ is adjacent to exactly two vertices among $x_1$, $x_2$, and $d_3$, which is impossible without creating a triangle.
Hence $b \adj d_2$.
It follows that $G_{y_1,d_2}$ is planar, except possibly when $d_3 \adj b,x_1,x_2$.
If $d_3 \adj b,x_1,x_2$, then $a \adj x_1,x_2$.
Moreover, $b$ is adjacent to either $w_1$ or $w_2$, and in either case $G_{x_1,x_2}$ is planar.
Thus $a \nadj y_1,y_2$, and similarly $b \nadj y_1,y_2$, so each of $a$ and $b$ is adjacent to both $d_1$ and $d_2$.
This implies that each of $a$ and $b$ is adjacent to exactly two vertices among $d_3$, $w_1$, $w_2$, $x_1$, and $x_2$.
Then $G_{d_1,d_2}$ is planar for every possible configuration, again a contradiction.

Next suppose, for a contradiction, that $w \adj d_1$ and $w_2 \adj d_2$.
If $a \adj y_1$, then $a \adj d_2$ and $b \adj d_1$.
Thus $a$ is adjacent to exactly two vertices among $x_1$, $x_2$, and $d_3$.
Then $G_{y_1,d_2}$ is planar, except possibly when $d_3 \adj b,w_1,x_1$.
If $d_3 \adj b,w_1,x_1$, then $a \adj x_2,d_3$.
Now there are two possibilities: either $b \adj w_1,d_2$, in which case $G_{y_1,d_2}$ is planar, or $b \adj w_2,y_2$, in which case $G_{y_1,d_3}$ is planar.
Thus $a \nadj y_1$.
If instead $a \adj y_2$, then $a \adj d_1$ and $b \adj d_1,d_2$.
If $d_3$ is not adjacent to both $a$ and $b$, then $G_{y_2,d_1}$ is planar.
Hence $d_3 \adj a,b$.
If $a \nadj w_2$, then $a$ is adjacent to one of $x_1$ and $x_2$, which implies that $G_{y_2,d_1}$ is planar.
Hence $a \adj w_2$.
Similarly, if $b \nadj x_2$, then $G_{y_2,d_1}$ is planar, so $b \adj x_2$.
But then $G_{y_2,d_3}$ is planar.
Thus $a \nadj y_2$.
Therefore $a \adj d_1,d_2$, and similarly $b \adj d_1,d_2$.
In this case, $a$ and $b$ can be adjacent only to vertices among $d_3$, $w_1$, $x_1$, and $x_2$.
We may assume that $a \adj w_1$, and then $b$ is adjacent to exactly two vertices among $d_3$, $x_1$, and $x_2$.
In any case, $G_{d_1,d_2}$ is planar.
This contradiction proves that $w$ is adjacent to neither $d_1$ nor $d_2$.

From the above, we conclude that $d_1 \adj w_1,y_1$ and $d_2 \adj w_2,y_2$.
If $d_3 \adj x_1,x_2$, then $G_{a,y_1}$ is planar.
Furthermore, if $d_3 \adj w$ and $a \nadj d_3$, then $G_{b,y_2}$ is planar.
So, without loss of generality, we may assume that $d_3 \adj a,b,w,x_1$.

Now suppose, for a contradiction, that $a \adj d_1,d_2$.
Then $a \adj x_2$, and so $b$ cannot be adjacent to both $d_1$ and $d_2$, since $x_2$ cannot support any further adjacency.
If $b \adj y_1$, then $b \adj d_2$, and hence $b \adj w_1$.
Then $G_{y_1,d_3}$ is planar.
If instead $b \adj y_2$, then $b \adj d_1$, and hence $b \adj w_2$.
Then $G_{y_2,d_3}$ is planar.
Therefore $a$ is not adjacent to both $d_1$ and $d_2$, and similarly neither is $b$.

Finally, we conclude that $a \adj y_1,d_2,d_3$ and $b \adj y_2,d_1,d_3$, where $d_1 \adj w_1,y_1$, $d_2 \adj w_2,y_2$, and $d_3 \adj w,x_1$.
Thus $a$ is adjacent to one of $w_1$ and $x_2$, and $b$ is adjacent to one of $w_2$ and $x_2$.
If $a \adj x_2$, then $G_{y_1,y_2}$ is planar.
If $b \adj x_2$, then $G_{y_2,d_3}$ is planar.
Hence $a \adj w_1$ and $b \adj w_2$.
The resulting graph is Cousin $43$ in the $E_9\!+\!e$ family.

\section{Degree sequence $[2,4,4]$} \label{sec:224}

In this section, we consider the third case, in which $G_{a,b}$ has degree sequence $[2,4,4]$.
Let $d_1$, $d_2$, $d_3$, and $d_4$ denote the degree-$2$ vertices of $G_{a,b}$.
As noted at the end of Section~\ref{sec:main}, we will show that the graphs $G$ yielding such a graph $G_{a,b}$ are precisely Cousins $43$, $105$, and $109$ in the $E_9\!+\!e$ family, and $H_{12}\!+\!e$ in the $H_9\!+\!e$ family.

After smoothing each degree-$2$ vertex, we obtain a nonplanar graph $\widehat{G}_{a,b}$ with two degree-$4$ vertices, four degree-$3$ vertices, and $10$ edges.
Because $\widehat{G}_{a,b}$ is nonplanar, it has a $K_{3,3}$ minor.
Hence there is an edge $e$ such that $K_{3,3}$ is obtained from $\widehat{G}_{a,b}$ by deleting $e$:
$$\begin{array}{ccc}
\widehat{G}_{a,b} \! & \! \rightarrow \! & \! \widehat{G}_{a,b} \setminus e\,  \simeq\, K_{3,3}
\end{array}$$

Let $K_{3,3}^+$ denote the graph obtained from $K_{3,3}$, with bipartition $\{x_1,x_2,x_3\}$ and $\{y_1,y_2,y_3\}$, by adding one extra edge joining $x_1$ and $x_2$.
Note that each endpoint of $e$ has degree $4$ in $\widehat{G}_{a,b}$.
According to the position of the endpoints of $e$, three cases arise:
\begin{enumerate}
\item[(a)] $e$ is a double edge, that is, $\widehat{G}_{a,b}$ is $K_{3,3}$ with double edges;
\item[(b)] $\widehat{G}_{a,b}$ is $K_{3,3}^+$, and the extra edge $e$ has some degree-$2$ vertices of $G_{a,b}$;
\item[(c)] $\widehat{G}_{a,b}$ is $K_{3,3}^+$, and the extra edge $e$ has no degree-$2$ vertex of $G_{a,b}$.
\end{enumerate}
These three cases are treated in the following three subsections.

\subsection{$\widehat{G}_{a,b}$ is $K_{3,3}$ with double edges} \label{subsec:Hdouble} \ 

Suppose that $\widehat{G}_{a,b}$ is obtained from $K_{3,3}$, with bipartition $\{x_1,x_2,x_3\}$ and $\{y_1,y_2,y_3\}$, by adding an edge $e$ connecting $x_1$ and $y_1$ as a double edge.
In this case, the resulting graph $G$ may be Cousin $43$ or $109$ in the $E_9\!+\!e$ family.
Let $e'$ denote the other double edge parallel to $e$.
Since $G$ is triangle-free, the two doubled edges $e$ and $e'$ together contain at least two degree-$2$ vertices of $G_{a,b}$, say $d_1$ and $d_2$.

Suppose that one of the vertices $x_2$, $x_3$, $y_2$, or $y_3$ is not adjacent to any degree-$2$ vertex of $G_{a,b}$; say $x_2$.
Since $G$ has maximum degree $4$, one of $a$ and $b$ is not adjacent to $x_2$; say $a \nadj x_2$.
Then $G_{b,y_1}$ is planar, a contradiction.
Therefore each of $x_2$, $x_3$, $y_2$, and $y_3$ is adjacent to a degree-$2$ vertex, namely one of $d_3$ and $d_4$.
Without loss of generality, we may assume that $d_3 \adj x_2, y_2$ and $d_4 \adj x_3, y_3$.

We now distinguish two cases.
First, suppose that one doubled edge, say $e$, contains both $d_1$ and $d_2$.
Then $a$ and $b$ must be adjacent to different vertices among $d_1$ and $d_2$.
Hence the remaining three neighbors of each of $a$ and $b$ lie on the $6$-cycle $(x_2d_3y_2x_3d_4y_3)$.
Without loss of generality, we may assume that $a \adj x_2,y_2,d_4$ and $b \adj d_3,x_3,y_3$.
Then the resulting graph $G$ is Cousin $43$ in the $E_9\!+\!e$ family.

Next, suppose that $e$ contains $d_1$ and $e'$ contains $d_2$.
If $a$ and $b$ are adjacent only to different vertices among $d_1$ and $d_2$, then, by the same argument as above, we may assume that $a \adj x_2,y_2,d_4$ and $b \adj d_3,x_3,y_3$.
But then $G_{x_1,y_1}$ is planar.
Therefore one of $a$ and $b$ must be adjacent to both $d_1$ and $d_2$; without loss of generality, assume that $a \adj d_1,d_2$.
If $a \adj x_2$ (and similarly if $a \adj x_3$, $y_2$, or $y_3$), then $G_{b,x_1}$ is planar.
Therefore $a \adj d_3,d_4$.

If $b \adj d_1,d_2$, then, as above, $b \adj d_3,d_4$, and it follows that $G_{x_1,y_1}$ is planar.
If $b \nadj d_1,d_2$, then $b$ has four neighbors on the $6$-cycle $x_2 d_3 y_2 x_3 d_4 y_3$, which creates a triangle.
Thus we may assume that $b \adj d_1$ and $b \nadj d_2$.
By symmetry, we may further assume that $b \adj x_2,y_2,d_4$.
Then the resulting graph $G$ is Cousin $109$ in the $E_9\!+\!e$ family.

\subsection{$\widehat{G}_{a,b}$ is $K_{3,3}^+$, and $e$ has 
a degree-$2$ vertex 
of $G_{a,b}$.} \label{subsec:Hplusdiv} \ 

Suppose that $\widehat{G}_{a,b}$ is $K_{3,3}^+$ and that the added edge $e$ contains a degree-$2$ vertex of $G_{a,b}$.
Without loss of generality, assume that $d_4$ lies on $e$.
In this case, the resulting graph $G$ may be Cousin $105$ or $109$ in the $E_9\!+\!e$ family, or $H_{12}\!+\!e$.

We first show that each vertex $y_i$ is adjacent to a degree-$2$ vertex of $G_{a,b}$.
Suppose, for example, that no degree-$2$ vertex is adjacent to $y_1$.
Since $G$ has maximum degree $4$, one of $a$ and $b$ is not adjacent to $y_1$; say $a \nadj y_1$.
Then $G_{b,x_1}$ is planar, a contradiction.
Hence, after relabeling if necessary, we may assume that $y_i \adj d_i$ for $i=1,2,3$.
Up to symmetry, we have six cases according to the placement of $d_1$, $d_2$, and $d_3$ as shown in Figure~\ref{fig:24416}.

\begin{figure}[h!]
\includegraphics[scale=1]{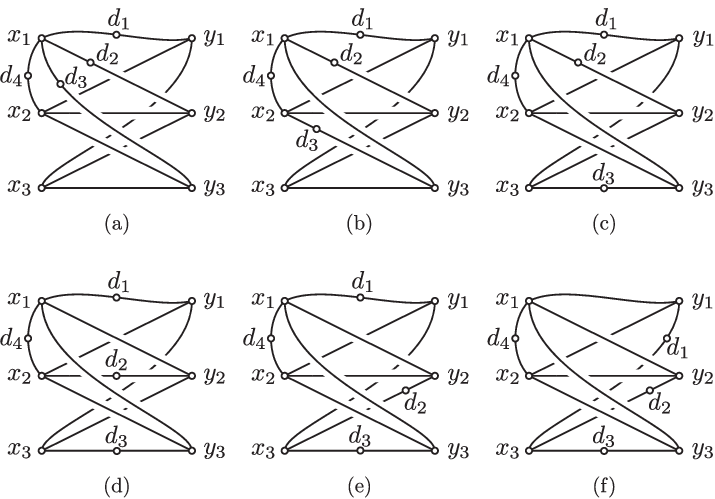}
\caption{Six cases with a vertex on edge $e$.}

\label{fig:24416}
\end{figure}

\noindent
\underline{Observation}:
each of $a$ and $b$ is adjacent to exactly one vertex from each pair $\{y_i,d_i\}$.
Otherwise, for example, if $a$ is adjacent to neither $y_1$ nor $d_1$, then again $G_{b,x_1}$ would be planar.\vspace{2mm}

\noindent{\bf Cases (a) and (b).} \
Suppose that $x_1 \adj d_1,d_2$ and either $x_1 \adj d_3$ (Case~(a)) or $x_2 \adj d_3$ (Case~(b)).
If $a \adj y_1,y_2$, then $a \adj d_4$, and so $G_{y_1,y_2}$ is planar, since the remaining neighbor of $a$ is either $d_3$ or $y_3$.
Thus each of $a$ and $b$ is adjacent to at most one of $y_1$ and $y_2$.
Using our observation, 
there are three cases:
$$ a,b \adj d_1,d_2; \qquad
a \adj d_1,d_2 \text{ and } b \adj y_1,d_2; \qquad
a \adj y_2,d_1 \text{ and } b \adj y_1,d_2.$$
In the first case, $G_{d_1,d_2}$ is planar.
In the second case, $b \adj d_4$, and the remaining neighbor of $b$ is either $y_3$ or $d_3$. 
If $y_3$, then $G_{y_1,y_3}$ is planar and, if $d_3$, then
one of $G_{y_1,d_3}$ or $G_{y_1,d_4}$ is planar.
In the last case, $a,b \adj d_4$, and so $G_{y_1,y_2}$ is planar.\vspace{2mm}

\noindent{\bf Case (c).} \
Suppose that $x_1 \adj d_1,d_2$ and $x_3 \adj d_3$.
We show that, in this case, $G$ is Cousin $105$ in the $E_9\!+\!e$ family.
By the same argument as in the previous case, each of $a$ and $b$ is adjacent to at most one of $y_1$ and $y_2$.
Hence, 
using our observation,
there are three cases:
$a,b \adj d_1,d_2$; $a \adj d_1,d_2$ and $b \adj y_1,d_2$; and $a \adj y_2,d_1$ and $b \adj y_1,d_2$.

In the first case, without loss of generality, assume that $a \adj d_3$.
Then $a \adj d_4$, and hence $G_{d_1,d_2}$ is planar.
In the second case, $b \adj d_4$, and the remaining neighbor of $b$ is either $y_3$ or $d_3$.
If $b \adj y_3$, then $G_{y_1,y_3}$ is planar.
If $b \adj d_3$ and $a \nadj x_3$, then $G_{y_1,d_2}$ is planar.
If $b \adj d_3$ and $a \adj x_3$, then 
also $a \adj y_3$, 
and
$G_{d_1,d_2}$ is planar.

It remains to consider the last case.
Since $x_3 \nadj a,b$, we have $d_4 \adj a,b$.
Without loss of generality, assume that $a \adj d_3$.
Then $b$ is adjacent to either $y_3$ or $d_3$.
If $b \adj y_3$, then $G_{y_1,y_3}$ is planar.
If $b \adj d_3$, then the resulting graph $G$ is Cousin $105$ in the $E_9\!+\!e$ family.\vspace{2mm}

\noindent{\bf Case (d).} \
Suppose that $x_i \adj d_i$ for $i=1,2,3$.
We show that $G$ is Cousin $105$ or $109$ in the $E_9\!+\!e$ family, or $H_{12}\!+\!e$.

First suppose that $a \adj x_3$; the case $b \adj x_3$ is symmetric.
Then, 
by our observation
and the triangle-free condition, we have $a \adj y_3,d_1,d_2$.
Hence $b \adj d_3,d_4$.
If $b$ is adjacent to at least one of $y_1$ and $y_2$, say $y_1$, then $G_{x_1,y_1}$ is planar.
Otherwise, $b \adj d_1,d_2$, and $G$ is Cousin $105$ in the $E_9\!+\!e$ family.

We may therefore assume that $a,b \nadj x_3$.
Then $a,b \adj d_4$.
First consider the case in which $a,b \adj d_1,d_2$.
If $a,b \adj d_3$, then $G_{d_1,d_3}$ is planar.
Otherwise, without loss of generality, assume that $a \adj d_3$ and $b \adj y_3$. Then $G$ is $H_{12}\!+\!e$.

It remains to consider the case in which one of $a$ and $b$ is adjacent to $y_1$ or $y_2$.
Without loss of generality, assume that $a \adj y_1$, and so $b \adj d_1$.
If $b$ is adjacent to $y_2$ or $y_3$, then $G_{y_1,d_4}$ is planar.
Hence $b \adj d_2,d_3$.
There are now four possibilities for the remaining two neighbors of $a$:
$$a \adj y_2,y_3;\qquad
a \adj y_2,d_3;\qquad
a \adj y_3,d_2;\qquad
a \adj d_2,d_3.$$
If $a \adj y_2,y_3$, then $G_{y_1,y_2}$ is planar.
If $a \adj y_2,d_3$, then $G$ is Cousin $109$ in the $E_9\!+\!e$ family.
If $a \adj y_3,d_2$, then $G_{y_1,y_3}$ is planar.
Finally, if $a \adj d_2,d_3$,
then $G_{d_2,d_3}$ is planar.\vspace{2mm}

\noindent{\bf Case (e).} \
Suppose that $x_1 \adj d_1$ and $x_3 \adj d_2,d_3$.
We show that $G$ is $H_{12}\!+\!e$.
First suppose that $a \adj x_3$.
Then $a \adj d_1,y_2,y_3$, and hence $G_{y_2,y_3}$ is planar.
Thus we may assume that $a,b \nadj x_3$.
This implies that $a,b \adj d_4$.

If $a \adj y_1$, then $G_{y_1,d_4}$ is planar.
Hence $a,b \adj d_1$.
Now suppose that $a \adj d_2,d_3$; the case $b \adj d_2,d_3$ is symmetric.
Then $G_{d_1,d_4}$ is planar.
Therefore the only remaining case is $a \adj y_2,d_3$ and $b \adj y_3,d_2$.
In this case, $G$ is $H_{12}+e$.\vspace{2mm}

\noindent{\bf Case (f).} \
Suppose that $x_3 \adj d_1,d_2,d_3$.
First suppose that $a \adj x_3$.
Then $a \adj y_1,y_2,y_3$ and $b \adj d_1,d_2,d_3,d_4$.
Then $G_{y_1,y_2}$ is planar.
Thus we may assume that $a,b \nadj x_3$.
This implies that $a,b \adj d_4$.
If $a,b \adj d_1,d_2,d_3$, then $G_{d_1,d_2}$ is planar.
Otherwise, without loss of generality, assume that $a \adj y_1$. Then $G_{y_1,d_4}$ is planar.

\subsection{$\widehat{G}_{a,b}$ is $K_{3,3}^+$, and $e$ has no degree-$2$ vertex of $G_{a,b}$.} \label{subsec:Hplusndiv} \ 

Suppose that $\widehat{G}_{a,b}$ is $K_{3,3}^+$ and that the added edge $e$ contains no degree-$2$ vertex of $G_{a,b}$.
We show that the resulting graph $G$ is Cousin $43$ in the $E_9\!+\!e$ family.
As in the previous subsection, each vertex $y_i$ is adjacent to a degree-$2$ vertex of $G_{a,b}$. 
Thus, we may say that $y_i \adj d_i$ ($i=1,2,3$). 
Let $d_4$ be the remaining degree-$2$ vertex.
Since the added edge $e$ is not subdivided, four of the six cases in Figure~\ref{fig:24416} contain triangles.
Therefore, up to symmetry, it is enough to consider the two triangle-free cases shown in Figure~\ref{fig:24422}.

\begin{figure}[h!]
\includegraphics[scale=1]{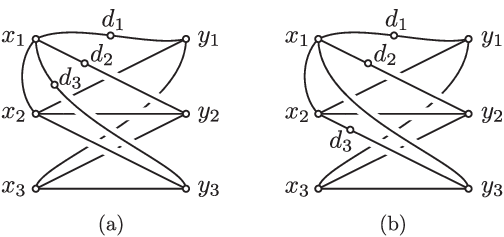}
\caption{Two triangle-free cases.}
\label{fig:24422}
\end{figure}

We claim that each of $a$ and $b$ is adjacent to at least one vertex from the set $\{y_i,d_i\}$ (or the set $\{y_i,d_i,d_4\}$ when $y_i \adj d_4$ or $d_i \adj d_4$) for $i=1,2,3$.
Indeed, suppose for example that $a$ is adjacent to none of the vertices in the relevant set for $i=1$.
Then, as before, $G_{b,x_1}$ is planar, a contradiction.\vspace{2mm}

\noindent{\bf Case (a).} \
Suppose that $x_1 \adj d_1,d_2,d_3$.
We show that $G$ is Cousin 43 in the $E_9\!+\!e$ family.
Without loss of generality, assume that $y_1 \adj d_4$.
If $x_2 \nadj d_4$, then $G_{x_2,d_2}$ is planar, since $y_2$ is adjacent to at most one of $a$ and $b$. 
Hence $x_2 \adj d_4$.
If $d_4$ is adjacent to at most one of $a$ and $b$, 
then $G_{x_2,d_2}$ is planar for the same reason.
Therefore $a,b \adj d_4$, and so $a,b \nadj y_1$.

First suppose that $a,b \adj d_2,d_3$.
If $a \nadj d_1$, then $a \adj x_3$ and $b \adj d_1$.
In this case, $G_{x_2,y_1}$ is planar.
Thus, up to symmetry, we may assume that $a,b \adj d_1$.
But then 
$G_{d_2,d_3}$
is planar.
Thus we may assume, for example, that $a \adj y_2$, and hence $b \adj d_2$.
If $d_3$ is adjacent to at most one of $a$ and $b$,
then $G_{y_3,d_4}$ is planar.
Therefore $a,b \adj d_3$.
Now, if $b \adj d_1$, then $G_{y_2,d_3}$ is planar.
Thus $b \adj x_3$, and consequently $a \adj d_1$. 
The resulting graph $G$ is Cousin 43 in the $E_9\!+\!e$ family.\vspace{2mm}

\noindent{\bf Case (b)-1.} \
Suppose that $x_1 \adj d_1, d_2$ and $x_2 \adj d_3$.
Additionally, suppose that $d_4$ is adjacent to some $d_i$ for $i=1,2,3$.
We show that $G$ is Cousin 43 in the $E_9\!+\!e$ family.
Up to symmetry, there are two cases: $d_4 \adj x_1,d_1$ and $d_4 \adj x_2,d_3$.

First suppose that $d_4 \adj x_1,d_1$.
Assume that $a \adj d_1$ and $b \adj d_4$.
If $a \nadj d_3$, then $G_{y_3,b}$ is planar.
So $a \adj d_3$, and similarly $b \adj d_3$.
Hence, $a,b \nadj y_3$, and consequently $a \adj x_3,d_2$ and $b \adj y_1$.
But then $G_{x_1,y_1}$ is planar.

Next suppose that $d_4 \adj x_2,d_3$.
Assume that $a \adj d_3$ and $b \adj d_4$.
By 
our observation,
there are three cases.
If $b \adj y_1,y_2$, then $G_{y_3,a}$ is planar.
If $b \adj d_1,d_2$, then $G_{d_1,d_2}$ is planar.
Thus, without loss of generality, we may assume that $b \adj y_1,d_2$.
Then $b \adj y_3$ and $a \adj x_3,d_1$.
If $a \adj y_2$, then $G_{x_1,y_1}$ is planar.
If $a \adj d_2$, then the resulting graph $G$ is Cousin 43 in the $E_9\!+\!e$ family.\vspace{2mm}

\noindent{\bf Case (b)-2.} \
Suppose that $x_1 \adj d_1, d_2$ and $x_2 \adj d_3$.
Additionally, suppose that $d_4$ is not adjacent to any $d_i$ for $i=1,2,3$.
We show that $G$ is Cousin 43 in the $E_9\!+\!e$ family.
Note that the case $d_4 \adj x_1,y_3$ is already covered in Case (a).
Up to symmetry, there are three cases:
$$ d_4 \adj x_2,y_1;\qquad d_4 \adj x_3,y_1;\qquad
d_4 \adj x_3,y_3. $$

First suppose that $d_4 \adj x_2,y_1$.
We show that $a,b \nadj y_i$ for $i=1,2,3$.
If $a \adj y_1$, then $a \nadj d_1,d_4,x_3$.
Thus the four adjacencies of $a$ cannot be completed.
Hence $a \nadj y_1$, and similarly $b \nadj y_1$.
Now suppose that $a \adj y_2$. 
Then $b \adj d_2$ and $a \adj d_1, d_4$.
If either $a$ or $b$ is not adjacent to $d_3$, then $G_{y_2,y_3}$ is planar.
Therefore $a,b \adj d_3$.
Furthermore, if $b$ is not adjacent to one of $d_1$ and $d_4$, then $G_{x_1,y_2}$ is planar.
On the other hand, if $b \adj d_1,d_4$, then $G_{d_3,d_4}$ is planar.
Thus $a,b \nadj y_2$, and hence $a,b \adj d_2$.
Similarly, suppose that $a \adj y_3$.
Then $b \adj d_3$ and $a \adj d_1,d_4$.
If $b \adj d_1$, then $G_{d_1,d_2}$ is planar.
Otherwise, $b \adj x_3,d_4$, and then 
$G_{d_2,d_4}$ 
is planar.
Thus $a,b \nadj y_3$, and hence $a,b \adj d_3$.
Consequently, each of $a$ and $b$ is adjacent to exactly two vertices among $x_3$, $d_1$, and $d_4$.
In every case, $G_{d_2,d_3}$ is planar.

Now consider the second case that $d_4 \adj x_3, y_1$.
Similarly, we show that $a,b \nadj y_i$ for $i=1,2,3$.
If $a \adj y_3$, then $G_{b,y_3}$ is planar.
Therefore $a,b \nadj y_3$, and hence $a,b \adj d_3$.
If $a \adj y_1$, then $a \nadj d_1,d_4$.
This implies that $a \adj x_3$ and $b \adj d_1,d_4$.
By the 
observation,
each of $a$ and $b$ is adjacent to either $y_2$ or $d_2$. 
If one of $a$ and $b$ is adjacent to $y_2$, then $G_{x_1,y_1}$ is planar.
Therefore $a,b \adj d_2$.
Then $G_{d_2,d_3}$ is planar.
Therefore $a,b \nadj y_1$.
If $a \adj y_2$, then $a \nadj d_2,x_3$.
This implies that $a \adj d_1,d_3,d_4$ and $b \adj d_2,d_3$.
Then $G_{y_2,d_4}$ is planar.
Thus $a,b \nadj y_2$.
Consequently, $a,b \adj d_1,d_2,d_3$, and each of $a$ and $b$ is adjacent to either $x_3$ or $d_4$.
In every case, $G_{d_2,d_3}$ is planar.

It remains to consider the case in which $d_4 \adj x_3, y_3$.
If $a \adj y_3$, then $G_{b,y_3}$ is planar.
Hence $a \nadj y_3$, and similarly $b \nadj y_3$.
Therefore $a,b \adj d_3$.
Suppose first that $a,b \nadj x_3$, which implies $a,b \adj d_4$.
If either $a$ or $b$ is adjacent to both $d_1$ and $d_2$, then $G_{d_3,d_4}$ is planar.
Thus we assume that $a \adj y_1,d_2$ and $b \adj y_2,d_1$.
But then $G_{x_1,y_1}$ is planar.
Therefore, without loss of generality, assume that $a \adj x_3$ and $b \adj d_4$, and hence $a \adj d_1,d_2$.
If $b \adj y_1,y_2$ or $b \adj d_1,d_2$, then $G_{d_3,d_4}$ is planar.
By symmetry of the resulting graph, the remaining neighbors of $b$ are $y_1$ and $d_2$.
The resulting graph $G$ is Cousin 43 in the $E_9\!+\!e$ family.

\section*{Acknowledgement}  
The first author was supported by the National Research Foundation of Korea(NRF) grant funded by the Korea government(MSIT) (RS-2026-25604998).
The third author was supported by the National Research Foundation of Korea(NRF) grant funded by the Korea government(MSIT) (No. RS-2024-00341075). 
This research was carried out, in part, during a visit of the second author to Kyungpook National University and he
thanks the department for their hospitality. We are grateful to Professors A.~Pavelescu and Teragaito for
valuable conversations.


\begin{thebibliography}{99}
\bibitem{BM} J. Barsotti and T.W. Mattman,
    {\em Graphs on 21 edges that are not $2$-apex},
    Involve  \textbf{9}  (2016) 591--621.
\bibitem{BBFFHL} P. Blain, G. Bowlin, T. Fleming, J. Foisy, J. Hendricks, and J. LaCombe,
    {\em Some results on intrinsically knotted graphs},
    J. Knot Theory Ramif. \textbf{16} (2007) 749--760.
\bibitem{CG} J. Conway and C. Gordon,
    {\em Knots and links in spatial graphs},
    J. Graph Theory \textbf{7} (1983) 445--453.
\bibitem{FMMNN} E. Flapan, T.W.~Mattman, B.~Mellor, R.~Naimi, and R.~Nikkuni,
 {\em Recent developments in spatial graph theory},
 Contemp. Math., \textbf{689}
American Mathematical Society, Providence, RI, (2017) 81–102.
\bibitem{F} J. Foisy,
    {\em Intrinsically knotted graphs},
    J. Graph Theory \textbf{39} (2002) 178--187.
\bibitem{GMN} N. Goldberg, T.W. Mattman, and R. Naimi,
   {\em Many, many more intrinsically knotted graphs},
    Algebr. Geom. Topol. \textbf{14} (2014) 1801--1823.
\bibitem{HNTY} R. Hanaki, R. Nikkuni, K. Taniyama, and A. Yamazaki,
    {\em On intrinsically knotted or completely $3$-linked graphs},
    Pacific J. Math. \textbf{252} (2011) 407--425.
\bibitem{JKM} B. Johnson, M. Kidwell, and T. Michael,
    {\em Intrinsically knotted graphs have at least $21$ edges},
    J. Knot Theory Ramif. \textbf{19} (2010) 1423--1429.
\bibitem{KLLMO} H. Kim, H. J. Lee, M. Lee, T.W. Mattman, and S. Oh,
    {\em A new intrinsically knotted graph with 22 edges},
    Topol. Appl. \textbf{228} (2017) 303--317.
\bibitem{KM} H. Kim and T.W. Mattman,  
{\em Dips at small sizes for topological graph obstruction sets},
Discrete Appl. Math. \textbf{360} (2025), 139–166.
\bibitem{KMO1} H. Kim, T.W. Mattman, and S. Oh,
    {\em Bipartite intrinsically knotted graphs with 22 edges},
    J. Graph Theory \textbf{85} (2017) 568--584.
\bibitem{KMO2} H. Kim, T.W. Mattman, and S. Oh,
    {\em More intrinsically knotted graphs with 22 edges and the restoring method},
    J. Knot Theory Ramif. \textbf{27} (2018) 1850059.
\bibitem{KS} T. Kohara and S. Suzuki,
    {\em Some remarks on knots and links in spatial graphs}
    in Knots 90 (de Gruyter, Osaka) (1992) 435--445.
\bibitem{LKLO} M. Lee, H. Kim, H. J. Lee, and S. Oh,
    {\em Exactly fourteen intrinsically knotted graphs have 21 edges},
    Algebr. Geom. Topol. \textbf{15} (2015) 3305--3322.
\bibitem{M} T.W. Mattman, 
   {\em Graphs of 20 edges are 2-apex, hence unknotted},
   Algebr. Geom. Topol. \textbf{11} (2011) 691--718.
\bibitem{OT} M. Ozawa and Y. Tsutsumi,
    {\em Primitive spatial graphs and graph minors},
    Rev. Mat. Complut. \textbf{20} (2007) 391--406.
\bibitem{RS} N. Robertson and P. Seymour,
    {\em Graph minors XX, Wagner's conjecture},
    J. Combin. Theory Ser. B \textbf{92} (2004) 325--357.
\bibitem{S} H. Sachs,
    {\em On spatial representations of finite graphs}
    in Colloq. Math. Soc. J\'{a}nos Bolyai \textbf{37} (North Holland, New York) (1984) 649--662.
\end{thebibliography}
\end{document}